\documentclass{amsart}  
\usepackage{appendix} 
\usepackage[colorlinks]{hyperref}        
\usepackage{amsmath, amssymb, mathrsfs, amsfonts,  amsthm} 
\usepackage[all]{xy}

\newcommand{\ham}{\Ham (S ^{2}, \omega)}

\newcommand{\hams}{\Ham (\Sigma, \omega)}

\DeclareMathOperator {\injrad}{injrad}

\DeclareMathOperator {\id}{id}

\DeclareMathOperator {\Ham} {Ham}

\DeclareMathOperator{\area}{area}

\DeclareMathOperator{\Vol}{Vol}
\DeclareMathOperator{\Diff}{Diff}

\newtheorem {theorem} {Theorem} [section]

\newtheorem {conjecture} [theorem] {Conjecture}
\newtheorem{lemma}[theorem] {Lemma}

\newtheorem {definition} [theorem] {Definition}

\newtheorem {corollary}[theorem]  {Corollary}

\newtheorem {notation}[theorem] {Notation}
\newtheorem {remark} [theorem] {Remark}
\numberwithin {equation} {section}

\begin{document} 
\author {Yasha Savelyev} 
\address{ICMAT, Madrid}
\email {yasha.savelyev@gmail.com}
\title [On the  injectivity radius conjecture]{On the Hofer geometry
injectivity radius conjecture}
\begin {abstract} 
 We verify here some variants of topological and dynamical flavor of  the injectivity radius
conjecture in Hofer geometry, Lalonde-Savelyev
\cite{citeLalondeSavelyevOntheinjectivityradiusinHofergeometry} in the case of $\ham$ and $\Ham(\Sigma,
\omega)$, for $\Sigma$ a closed positive genus  surface. 
 In
 particular we show that any loop in $\ham$, respectively $\hams$ with
$L ^{+}$ Hofer length less than
 $\area(S ^{2} )/2$, respectively any $L ^{+} $ length is
 contractible through ($L ^{+} $) Hofer shorter loops, in the $C ^{\infty} $
 topology. We also prove some stronger
variants of this statement on the loop space level. One dynamical type
corollary is that there are no smooth, positive Morse index
(Ustilovsky)
geodesics, in    
$\ham$, respectively in $\hams$ with $L ^{+} $ Hofer length less
than $\area (S ^{2} )/2$, respectively any length. The above condition on the
geodesics can be expanded as an explicit and elementary dynamical condition on the
associated Hamiltonian flow. We also give some speculations on
connections of this later result with curvature properties of the
Hamiltonian diffeomorphism group of surfaces. 
\end {abstract}
\maketitle
\section {Introduction}
One of the most fundamental objects associated to a Finsler manifold
is its injectivity radius function.  The 
group $\Ham(M, \omega)$, with its bi-invariant Finsler Hofer
metric, has no
well defined exponential map, so that if we ask about its injectivity
radius we need to decide how to interpret this. One simple way to
interpret is to ask
 for the size of the largest
metric epsilon-ball which is contractible, or has
null-homotopic inclusion map into the total space $\Ham(M,
\omega)$. 

In this formulation we studied the question of injectivity
radius for $\Ham(M, \omega)$ in Lalonde-Savelyev
\cite{citeLalondeSavelyevOntheinjectivityradiusinHofergeometry}, and
formulated there the conjecture that the injectivity and (or) weak injectivity
radius of $\Ham(M, \omega)$ is positive. 

In this note we shall verify some variants of this conjecture in the
case of surfaces. 
The main new ingredient
is a kind of ``curvature'' \footnote {This is a slight misnomer as curvature
may not be monotonically decreasing under flow. But probably the best
way to think about it.} flow for connection type symplectic forms on surface Hamiltonian fibrations over $S ^{2}
$, which exists in the presence of certain foliations by holomorphic curves,
using which we also obtain a curve shortening algorithm. It is interesting to
speculate whether this is at all related to symplectic Ricci type
flows, but to 
emphasize, our ``flow'' is of very elementary nature, at least assuming
the state-of-the-art in Gromov-Witten theory.

Using these developments
we obtain  some dynamical applications, which are formulated in terms of
certain
non-existence results of Ustilovsky geodesics. These
may also be interpreted as injectivity radius statements but from a
Finsler geometric rather than purely metric point of view.   This
also uses the author's (virtual) Morse theory for the Hofer length
functional. 
We follow this with some speculations on the ``curvature'' of the Hamiltonian group of surfaces.

\subsection{Statements}
\label{sec:statements}
Given a smooth function $H: M ^{2n} \times (S ^{1} =
\mathbb{R}/ \mathbb{Z}) \to \mathbb{R}$, (time is in $S ^{1} $ for
convenience and later use)  there is  an associated time-dependent
Hamiltonian vector field $X _{t}$, $0 \leq t \leq 1$,  defined by 
\begin{equation}  \label{equation.ham.flow} \omega (X _{t}, \cdot)=-dH _{t} (
\cdot).
\end{equation}
The vector field $X _{t}$ generates a path $\gamma: [0,1] \to \text
{Diff}(M)$, starting at $\id$. Given such a path $\gamma$, its end point
$\gamma(1)$ is called a Hamiltonian symplectomorphism. The space of Hamiltonian symplectomorphisms
forms a group, denoted by $\Ham(M, \omega)$.
In particular the path $\gamma$ above lies in $ \Ham(M, \omega)$. 
Given  a general smooth path $\gamma$, the 
\emph{positive Hofer
length}, $L ^{+}  (\gamma)$ is defined by 
\begin{equation*} L ^{+}  (\gamma):= \int_0 ^{1} \max _{M} H _{t} ^{\gamma} 
dt,
\end{equation*}
where $H ^{\gamma}$ is a generating function for the path
$t \mapsto \gamma({0}) ^{-1} \gamma (t),$ $0\leq t \leq 1$, normalized to have zero
mean at each moment. 

Part of the reason for interest in the above functional as well as in the
original Hofer length functional, which is obtained by integration of
the $L ^{\infty} $ norm of the generating function, is due to existence of some
deep connections of this with both the theory of dynamical
systems and with Floer-Gromov-Witten theory.


%
\begin{notation} 
Let $$\Omega ^{c} \Ham (M, \omega):= \{\gamma \in \Omega
   \Ham (M, \omega) \, | \, L ^{+} (\gamma) <c \}.$$ This is taken in the topology induced by the $C
^{\infty} $ topology. 
\end{notation}
\begin{definition}
  Let 
   $$\injrad _{\Omega,+} (M, \omega):= \sup \{c \, | \, \Omega ^{c'}
   \Ham (M, \omega) \text{ is (intrinsically) contractible for $c' \leq c$}\}. $$
\end{definition}
\begin{theorem}     
Set $\hbar = \area (S ^{2},\omega )$, and let $\Sigma$ denote a positive genus
surface,
then we have
\begin{align} 
\label{eq:equality}
&  \injrad _{\Omega, +}(S ^{2} )= \hbar/2, \\ 
   &  \injrad _{\Omega, +} (\Sigma)= \infty.
\end{align}
  \end{theorem}
  \begin{corollary}
In particular a smooth loop $\gamma \in \Omega ^{\hbar/2} \ham$,
respectively in $\Omega \hams$ with any $L^{+} $ Hofer length, 
is  contractible through Hofer shorter loops. 
  \end{corollary}
The above theorem is implied by the following theorem once we 
observe that  any loop with positive Hofer length less
than $\hbar/2$ is contractible, which is a classical fact proved for
example via
Seidel morphism
\cite{citeSeidel$pi_1$ofsymplecticautomorphismgroupsandinvertiblesinquantumhomologyrings}.
\begin{theorem} \label{thm:contractibleloops} 
Denote by 
$$\Omega
^{c, cont} \ham \subset \Omega ^{c} \ham $$ the subspace of
contractible loops,
then $\Omega
^{c, cont} \ham $ is contractible for any $c \leq \hbar$.
\end{theorem}
  We can readily see that this bound $\hbar$  is optimal,
as we have a natural representative 
$f _{\min} $
of a generator of $$\pi _{2} (\Omega \ham) = \pi _{3} (\ham) \simeq
\mathbb{Z}  $$ all of whose loops are
contractible and have $L ^{+} $
Hofer length at most
$\hbar$. Let us make this more explicit. Take the natural Lie group homomorphism $\widetilde{f}: S
^{3} \to \ham,$ representing the generator of $\pi _{3} (\ham) $. Deloop this
to a map $f _{\min} : S ^{2}  \to \Omega \ham $, by taking the natural $S ^{2} $ family 
of based at $\id$ loops on $S ^{3} $, so that the longest loop in this
family corresponds to a simple great geodesic $\gamma$, for the round
metric. In other words this is
the family 
forming the unstable manifold for  $\gamma$, which is an index 2
geodesic.
Then $f _{min}  (\gamma)$  is the longest loop
in the image of $S ^{2} $, and its positive Hofer length
is $2 \hbar/2 =\hbar$.
Consequently
we get: 
\begin{corollary} 
  The map $f _{\min} $  cannot be homotoped into $\Omega ^{\hbar} \ham
  $.
\end{corollary}
We also proved this using a  more general theory of quantum
characteristic classes in 
 \cite{citeSavelyevQuantumcharacteristicclassesandtheHofermetric.}.
%
\subsubsection {Dynamical consequences} \label{section:dynamical}
It is shown by Ustilovsky
\cite{citeUstilovskyConjugatepointsongeodesicsofHofersmetric} that $ \gamma$ is a smooth critical
point of $L ^{+} $ Hofer length functional on the smooth path space
from $\id$ to $\phi$, 
if there is a  point $x _{\max} \in M$
maximizing the generating function $H ^{\gamma} _{t}$
at each moment $t$, and such that $H ^{\gamma} _{t}$ is Morse at $x
_{\max}$, at
each moment $t $.
 \begin{definition} We call such a $\gamma$ an  \emph {
\textbf{$L ^{+}$ Ustilovsky geodesic}}. 
\end{definition}
 As an Ustilovksy geodesic is a smooth critical point  it makes sense
 to ask for its Morse index.  The index was shown by the author in
 \cite{citeSavelyevProofoftheindexconjectureinHofergeometry} to be
 finite, in fact for closed Ustilovsky geodesics in the group of
 Hamiltonian diffeomorphisms of a surface we have:
\begin{theorem} \label{thm:indexUst}
Let $\gamma \in \Omega \Ham (M, \omega )$ 
be a smooth closed $L ^{+} $ Ustilovsky geodesic, where $M$ is a
surface. Let $\gamma_* $
denote the linearization of $\gamma $ at $x _{\max} $, which is a loop
of linear symplectomorphisms of $T _{x} M  $.
 Then the Morse index of $ \gamma$ with respect to
 $L^{+}$ (and conventions \eqref{eq:conventions}) is
 \begin{equation*}
-Maslov (\gamma _{*} ) - 2,
 \end{equation*}
 if $Maslov (\gamma _{*} ) \leq -2$, otherwise the Morse index is 0.
\end {theorem} 
Here the Maslov number is normalized so that for the clockwise single 
rotation of $\mathbb{R}^{2} $ the Maslov number is -2. The above
theorem can be readily deduced from 
   \cite{citeSavelyevProofoftheindexconjectureinHofergeometry}.
\begin{definition}  We will say that a smooth closed Ustilovsky
   $L ^{+} $ geodesic  $$\gamma: S ^{1} \to \Ham(M, \omega)  $$ is \emph{\textbf{quasi-integrable}}, if there
   is a Darboux chart $\phi: (U \subset \mathbb{C}^{n})  \to M $, $\phi (0) = x _{
   \max} $ at the extremizer $x
   _{\max}   $, in which the generating function $H ^{\gamma} $
   coincides with its Hessian quadratic form in some neighborhood of
   $0$, and this quadratic form is the real part of a complex
   quadratic form, or in other words if the Hamiltonian flow for this
   Hessian quadratic form is unitary. 
 \end{definition} 
 Thus this is a kind of Morse-Darboux integrability condition at $x
 _{max} $, and is automatic when $\gamma $ is a circle action near $x
_{\max} $.
\begin{theorem} \label{thm:dynamical}
There are no non-constant, smooth,  
closed, positive Morse index $L ^{+} $ Ustilovsky geodesics, in $\ham$,
respectively in $\hams$, with $L ^{+} $ Hofer length less than
$\hbar/2$, respectively any length. 

The same statement holds for index
0 geodesics if they are \emph{quasi-integrable Ustilovsky}.

%
%
%
\end{theorem}
This will be proved in Section \ref{section:Ustilovsky}.
The positive Morse index condition on Ustilovsky geodesic is in a
sense an elementary dynamical condition.
So the previous theorem is really an elementary statement in
Hamiltonian dynamics.
It would be interesting to understand how it
extends to more general geodesics, or even what the appropriate
generalization of Ustilovsky geodesics is. For example McDuff and Lalonde
consider certain generalizations of Ustilovsky geodesics in
\cite{citeLalondeMcDuffHofers$Linfty$--geometry:energyandstabilityofHamiltonianflowsIII},
but these seem to be ill adapted to the variational style arguments
that we make. We have also avoided talking about non closed geodesics,
for while Savelyev~\cite{citeSavelyevMorsetheoryfortheHoferlengthfunctional} is
well adapted to the case of paths, our arguments here are not, and do
not immediately generalize.
\begin{remark}
 Very heuristically the
 statement for $\Ham(\Sigma, \omega)$ suggests that the
 ``curvature'' of $\Ham(\Sigma, \omega)$ with respect to the positive
 Hofer length functional is non-positive. In a finite dimensional
 setting this could be justified via Gauss-Bonnet theorem.
 See also Milnor's
 \cite{citeMilnorMorsetheory} for background on connections of
 curvature and topology of loop spaces. Similarly we may interpret the
 statement for $\ham$ as a certain upper bound on the positivity of
 ``curvature''.  However what is
 ``curvature''? We can try to think of this in terms of coarse
 geometry, the dream is something like:
\begin{conjecture} (Preliminary)
The space   $\Ham(T ^{2} ,
 \omega)$ is a rough $CAT (0)$ space, and $\Ham(\Sigma _{g}  ,
 \omega)$ is a rough $CAT (k)$ space for $k \leq 0$, for $g>1$, while  $\ham$ is a rough
 $CAT (k)$ space for some $k>0$, see \cite{citeRoughCat0}, with respect
 to the positive Hofer length functional.
\end{conjecture} 
\end{remark} 

\section{Proof of Theorem \ref{thm:contractibleloops}}
We first verify the case of $M=S ^{2} $. 
The following argument works
the same way for any $0 < c \leq \hbar$, 
we shall just do it with
$c=\hbar$. 
To recall, our conventions for the Hamiltonian flow and compatible almost complex
structures are:
\begin{align} \label{eq:conventions}
\omega (X _{H}, \cdot) = - dH (\cdot) \\
\omega (v, J v) > 0, \text{ for $v \neq 0$ }.
\end{align}
For every $l \in \Omega ^{\hbar, cont} \Ham(S ^{2} )   $ we get a  Hamiltonian $S ^{2} $  fibration $X _{l} $ over
$\mathbb{CP}^{1} $, by using $l$ as a clutching map: 
\begin{equation*}
X _{l} = S ^{2} \times D ^{2} _{-} \sqcup  S ^{2} \times D ^{2} _{+}/
\sim.
\end{equation*}

Recall that a \emph{coupling form}, see
\cite{citeGuilleminLermanEtAlSymplecticfibrationsandmultiplicitydiagrams}
for a Hamiltonian fibration $M
\to  P \overset {\pi}\hookrightarrow X$ is a closed 2-form $\widetilde{\alpha} $ on the total space, which
restricts to the symplectic form on the fibers, modeled by $(M,
\omega)$, and which satisfies:
\begin{equation*}
\pi _{*}  \widetilde{\alpha}  =0,
\end{equation*}
where $\pi _*$  is the integration over the fiber map.
By
\cite{citeGuilleminLermanEtAlSymplecticfibrationsandmultiplicitydiagrams}
such  forms $\widetilde{\alpha} $ always exist for a Hamiltonian fibration, and are
uniquely determined by the associated Hamiltonian connections, defined 
 by declaring the horizontal
subspaces to be the $\widetilde{\alpha} $-orthogonal subspaces to the
vertical tangent spaces.
We have the coupling form $\widetilde{\Omega} _{l,-}  $ on $S ^{2}
\times D ^{2} _{-}   $ defined by 
\begin{equation*}
\widetilde{\Omega}_{l,-}  = \omega -d (\eta (r) \cdot H ^{l}d \theta),
\end{equation*}
where $0 \leq r \leq 1,  0 \leq \theta \leq 1 $, (which are our modified
polar coordinates) $H ^ {l} $ is the normalized generating function
for $l$, and
$\eta: [0,1] \to [0,1]$ is a smooth  function
satisfying:
$$0 \leq \eta' (r),$$ and 
\begin{equation} \label{eq.eta} \eta (r) = \begin{cases} r^2 & \text{if }
0 \leq r \leq 1- 2\kappa \\ 1 & \text{if }  1 -\kappa \leq r \leq 1,
\end{cases}
\end{equation}
for a small $\kappa >0$. 
Under the gluing relation $\sim$, $\widetilde{\Omega} _{l, -}  $ corresponds to the
form $\omega$ near the boundary of $S ^{2} \times D ^{2}_{+}   $, so we may extend trivially over
$D ^{2} _{+}  $ to get a coupling form $\widetilde{\Omega} _{l}  $ on $X _{l} $.
The coupling form $\widetilde{\Omega}_{l}  $ determines a Hamiltonian connection
$\mathcal{A}_{l} $ as described above. 
This in turn determines an almost complex structure $J _{l}$, by first fixing a family $\{j _{l}(z)\}$, $z
\in \mathbb{CP}^{1} $ smooth in $z$: $j _{l}
(z)$ is an almost complex structure on the fiber $\pi _{l} ^{-1} (z)
$, where $$\pi _{l}: X _{l} \to \mathbb{CP}^{1}   $$ is the
projection, with each $j _{l} (z) $ compatible with the symplectic
form on the fiber, and then defining $J _{l} $  to coincide with $\{j
_{l} (z)\}$
on the vertical tangent bundle, to preserve the horizontal distribution of
$\mathcal{A} _{l} $  and to have a holomorphic projection map to
$(\mathbb{CP}^{1},  j) $, where $j$ is the almost complex
structure which preserves the standard orientation on
$\mathbb{CP}^{1} $.
As each $l$ is contractible and so $X _{l} $ trivializable, we may consider the moduli spaces
$$\overline{ \mathcal{M}}(J _{l }), $$ consisting of  stable
 $J _{l } $ holomorphic  sections $\sigma$ of $X _{l}$, in the class of
the constant section $[const]$, and in particular satisfying:
\begin{equation*}
   \langle [\widetilde{\Omega}_{l}]  , [\sigma]  \rangle = 0.
\end{equation*}
By a \emph{stable holomorphic section} we mean a stable $J _{l} $-holomorphic
map $\sigma$ into $ X _{l} $, in the classical sense,
\cite{citeMcDuffSalamon$J$--holomorphiccurvesandsymplectictopology} with domain an umarked nodal Riemann sphere, one
of those components is called \emph{principal}. The restriction $\sigma
_{princ} $ of $\sigma$, to the principal component is a $J _{l} $
holomorphic section, i.e. we have a commutative diagram:
\begin{equation*}
\xymatrix {& X _{l} \ar [d] ^{\pi _{l} }  \\ 
\mathbb{CP}^{1} \ar [ur] ^{\sigma _{princ} } \ar [r] ^{\id}  &   \mathbb{CP}.}
\end{equation*}
All the other components of $\sigma$ are vertical, that is they are
$J_l$-holomorphic maps into the fibers of $X _{l} $.
\begin{lemma}
The moduli space $\overline{ \mathcal{M}} (J _{l } )  $  has no nodal
curves as elements. 
\end{lemma}
\begin{proof}
Suppose otherwise, then there is a stable $J _{l } $-holomorphic section $\sigma$ of $X _{l} $, with
total homology class $[const]$, and consequently having a principal component
$\sigma _{princ} $ which is a smooth $J _{l} $-holomorphic section of $X _{l} $
with:
\begin{equation*}
 \langle [\widetilde{\Omega}_{l}  ], [\sigma _{princ} ]  \rangle \leq -\hbar
\end{equation*}
as $\hbar$ is the minimal energy of a non-constant holomorphic sphere in $S
^{2} $, i.e. $\area (S ^{2},\omega )$. However in this case the classical energy inequality for holomorphic
curves gives:
\begin{equation*}
\hbar \leq \area ^{+}  (\widetilde{\Omega}_{l}),
\end{equation*}
where $\area ^{+} $  is the functional on the space of coupling forms
\begin{equation*}
   \left. \area ^{+}  (\widetilde{\Omega}) = \inf _{\alpha}  \left \{\int
         _{\mathbb{CP}^{1}    } \alpha \, \right| \,
     \widetilde{\Omega}  + \pi ^{*} (\alpha) \, \text { is
 symplectic} \right \},
\end{equation*}
for $\alpha$ a $2$-form on $\mathbb{CP}^{1} $, with positive integral.
On the other hand by direct calculation we have 
$$ \area ^{+}  (\widetilde{\Omega}_{l}
) = L ^{+}   (l) < \hbar.$$ This gives a contradiction.
\end{proof}
%

By automatic transversality see
\cite[Appendix C]{citeMcDuffSalamon$J$--holomorphiccurvesandsymplectictopology}
$\overline{ \mathcal{M}} (J _{l } )  $  is
regular i.e. the associated real linear Cauchy-Riemann operator is transverse
for all $\sigma \in \overline{ \mathcal{M}} (J _{l } ) $. As these
are embedded we may use 
positivity of
intersections,  see \cite [Section
  2.6]{citeMcDuffSalamon$J$--holomorphiccurvesandsymplectictopology},
to infer that $\overline{
\mathcal{M}} (J _{l })$  determines a smooth  foliation of $X _{l} $ by
$[const]$ class holomorphic sections.  In particular
\begin{equation*}
   ev: \overline{ \mathcal{M}} (J _{l } )  \to S ^{2},
\end{equation*}
obtained by evaluating a section at $0 \in \mathbb{CP}^{1} $, is 
a diffeomorphism,  and determines   a canonical
smooth trivialization of $ X _{l} $. Let $\Theta _{l } $ denote the
corresponding horizontal distribution.

Consequently for an appropriately smooth family:
\begin{equation*}
   f: S ^{k}  \to \Omega ^{\hbar, cont} \ham,
\end{equation*}
we get  natural (up to  choices of almost complex structures)
smooth trivialization of the bundle 
\begin{equation*}
P _{f}  \to S ^{k},
\end{equation*}
with fiber over $s \in S ^{k} $ being $X _{s} \equiv X _{f (s)} $. 
This is a 
trivialization of a bundle with structure group $\Omega ^{2}  \Diff (S
^{2}) $. What is important for what follows  is that this structure
group is a subgroup of the group of smooth
bundle maps of $S ^{2} \times \mathbb{CP}^{1} \to \mathbb{CP}^{1}    $.

Let 
\begin{equation*}
tr: P _{f}  \to (S ^{2} \times \mathbb{CP}^{1}  ) \times S ^{k},  
\end{equation*}
denote this trivialization. Set $\{\mathcal{F}_{s}= (tr _{s})
^{*}   
\omega \}$, where $tr_{s} $ denotes the restriction of $tr$ to the
fiber $
{X _{s} }$. By the above $tr _{s} $ is a smooth bundle map and hence
each $\mathcal{F}_{s} $ is a \emph{connection type} closed form, which
just means the restriction to each fiber of $\pi_s: X _{s} \to \mathbb{CP}^{1}
$, is symplectic. This in turn 
 means that the smooth bundle connection determined
 by declaring horizontal spaces to be $\mathcal{F}_{s} $-orthogonal
 spaces to the vertical tangent spaces, is Hamiltonian (with respect
 to the family of fiber symplectic forms determined by the
 restrictions.)

Also by construction each
$\mathcal{F} _{s} $ vanishes on the horizontal distribution $\Theta _{f
(s)} $, and so 
\begin{align*}
& \tau _{s}   = \mathcal{F} _{s} + \rho \pi ^{*}_{s}  \tau, 
\end{align*}
is a  symplectic form on $X _{s} $ for any small $\rho>0$, and where $\tau$ a fixed, area
one symplectic form on
$\mathbb{CP} ^{1} $. Likewise
\begin{align*}
   & \alpha _{s}= \widetilde{\Omega}_{f (s)} + \area ^{+}
   (\widetilde{\Omega}_{f(s)}  )    \pi _{s} 
^{*} \tau  + 
\rho \pi ^{*}_{s}  \tau,
\end{align*}
is symplectic,
and $\alpha _{s} $ is positive on the horizontal
distribution $\Theta _{f (s)} $. 

Let $\alpha _{s,t}  $, $t \in [0,1]$  denote
the convex linear combination 
\begin{equation*} \label{eq:convex}
   \alpha _{s,t} = (t) \tau _{s}   + (1-t) \alpha
_{s}.
\end{equation*}
Note that $\alpha _{s,t
  } $ is nondegenerate and hence symplectic for every $s,t$, as this is a convex linear
  combination of symplectic forms tamed by the same almost complex
  structure $J(
  \mathcal{A} _{s} )
 $, where 
 $\mathcal{A} _{s} $ is the connection determined by
 $\widetilde{\Omega}_{f (s)}$ as before,
and so $\alpha _{s,t
  }$ is also tamed by this complex structure. Moreover for a fixed
  fiber $S ^{2} _{z} $, $\alpha _{s,t
  } $ is a convex linear combination of cohomologous symplectic forms
  $\mathcal{F}_{s}| _{S ^{2}_{z}  }  $ and
  $\widetilde{\Omega}_{f (s)} | _{S ^{2} _z }  $. By construction
  $(S ^{2}_0, \widetilde{\Omega}_{f (s)} | _{S ^{2} _0 } ) $ is naturally
  symplectomorphic to $(S ^{2}, \omega) $, consequently applying
  Moser's argument  we get that
   $\alpha _{s,t
  }|_{ S ^{2}_{0}}  $ is naturally symplectomorphic to $(S ^{2},
  \omega) $, for each $t$. We shall use this further on in the
  ``radial trivialization''construction.
\begin{notation}
Denote by $\widetilde{\alpha} _{s,t}  $ the coupling form determined
by $\alpha _{s,t} $.
\end{notation}
\begin{lemma} \label{lemma:bounded} $\area ^{+} (\widetilde{\alpha}  _{s,t})$ is
bounded from above by a function $b (s,t)$  non-increasing with $t$,
with $b (s, 1)=\rho$, and with $b (s, 0)= L ^{+} (f (s))+ \rho $.
\end{lemma}
\begin{proof}
Define   
\begin{equation} \label{eq:volumelength}
\area   (\alpha) := \Vol (X _{s}, \alpha)/ \Vol
(S ^{2} ,\omega).
\end{equation}
Note that
   $$
   \area ^{+} (\widetilde{\alpha} _{s,t})  \leq  \area (\alpha
   _{s,t}),$$
which follows by the fact that $$\alpha_{s,t} = \widetilde{\alpha}  _{s,t} + \pi _{s}
^{*} \tau',  $$  for an area form $\tau'$ by the construction.
Set $b (s,t)= \area (\alpha _{s,t} )$, then by direct calculation
$$b (s,0)= L ^{+} (f (s)) + \rho,  $$ and $$b (s,1) = \rho.$$
\end{proof}

Each $\widetilde{\alpha} _{s,t} $ determines a loop $f _{s,t} $ in
$\ham$, first by using the natural identification of $(S ^{2}_0,
\widetilde{\alpha}_{s,t}| _{S ^{2}_0 })    $ with $(S ^{2}, \omega)
   $ as described above, then
identifying the fiber over $0 \in \mathbb{CP}^{1} $ with the fiber over $
\infty \in \mathbb{CP}^{1} $ by $\widetilde{\alpha} _{s,t}$-parallel transport map over the
$\theta=0$ ray from $0$ to $\infty$ ($0 = 0 \in D ^{2}_{-}$, $\infty = 0 \in D
^{2} _{+}$ in our coordinates from before), and then
for every other $\theta$ ray from $0$ to $\infty$ getting an element $f _{s,t}
(\theta) \in \ham$ by $\widetilde{\alpha} _{s,t} $ parallel transport
over this ray. Clearly  $f _{s,0} = f (s) $, and $f _{s,1}=\id $ and
for each $t$, $f (s,t)$ is a loop of Hamiltonian symplectomorphisms of
$(S ^{2}, \omega) $ (by the identification above.)
On the other hand by the proof
of 
  \cite [Lemma
  3.3A]{citePolterovichSymplecticaspectsofthefirsteigenvalue.}  (cf. proof
  of
  \cite [Lemma 2.2]{citeMcDuffGeometricvariantsoftheHofernorm}) we have:
\begin{equation} \label{eq:Polterov}
   L ^{+}  (f _{s,t} ) \leq  \area ^{+}  (\widetilde{ \alpha} _{s,t} ).
\end{equation}

Consequently, by the Lemma \ref{lemma:bounded}  we get that
$L ^{+} (f _{s,t}) $ as a function of $t$ is bounded from
above by the
non-increasing continuous function $b (s,t) $ with maximum $L
^{+} (f(s))+\rho$, and with $b (s,1) =\rho$. Taking $\rho$ small
enough so that $L ^{+} (f _{s} ) + \rho < \hbar $ for each $s$, we obtain a null-homotopy  of $f$ in $\Omega
^{\hbar, cont} \ham $.

So we get that $\Omega
^{\hbar, cont}\ham $
is weakly contractible. 
    Next note that $\Omega
^{\hbar} \ham$ is an open subset of $\Omega \ham$, which by Milnor
\cite{citeMilnoruniversalbundles}, has the
homotopy type of a CW complex, as it is the loop space of a Frechet
manifold, hence of an absolute neighborhood retract, and so $\Omega
^{\hbar} \ham$ has the
homotopy type of a CW complex. So it follows that $\Omega ^{\hbar,
cont}
\ham $ is contractible.



For a $\Sigma$ as in the statement of the theorem, we may proceed with
exactly the same argument upon noting that now there is no bubbling at
all, so that we don't need to restrict to short loops.
\qed
\section {A virtually perfect Morse theory for the Hofer length
functional and positive Morse index geodesics in $\ham, \hams$} \label{section:Ustilovsky}

Let us review our notion of a local unstable manifold in the setting
of loops.  
\begin{definition} \label{def.unstable} Let $\gamma \in \Omega 
\Ham(M, \omega)$ be
an index $k$, $L ^{+} $ Ustilovsky geodesic and let $B ^{k} $ denote the standard
 $k$-ball in $ \mathbb{R} ^{k} $, centered at the origin 0. Let $
\Omega \Ham(M, \omega) _{E}  $ denote the  $ E
$ sub-level set of the loop space,  with respect to $L ^{+} $, with $$0 < E  < L ^{+}
   (\gamma),$$ where by $E$ sub-level set we mean the set of elements
   $\gamma \in \Omega \Ham(M, \omega) _{E} $  satisfying $L ^{+}
   (\gamma) \leq E $. \emph {\textbf{A local unstable manifold}} at $\gamma $ is a
   pair $ (f _{\gamma}, E ) $, with  $$f _{\gamma}: B ^{k} \to \Omega
   \Ham(M, \omega),$$ s.t. $$f _{\gamma} (0)  = \gamma,$$ $f _{\gamma} ^{*} L ^{+} $ is
   a function Morse at the unique maximum $0 \in B ^{k}$, and s.t. $$f _{\gamma}
   (\partial B ^{k} ) \subset \Omega  \Ham(M, \omega) _{E} .$$ \end{definition}
It is explained in
\cite{citeSavelyevMorsetheoryfortheHoferlengthfunctional} that local
unstable manifolds always exist.

An interesting consequence of the ``index theorem'' is that for
closed Ustilovsky geodesics the Morse index must be even, and this gives a bit
of intuition into the following theorem which can be interpreted as
saying  that 
the Hofer length functional on the loop space of the Hamiltonian group
of a surface is a kind of perfect Morse-Bott function. We say surface
for while the index stays even in general the following theorem does
not hold without an additional technical assumption on the geodesic.
(We need the kernel of certain Cauchy-Riemann operator associated to
the geodesic to be trivial.)

 The
  following is a version of the author's
  \cite [Theorem 1.9]{citeSavelyevMorsetheoryfortheHoferlengthfunctional}, in
  the case of surfaces, where it becomes a stronger result,
  due to the additional input of automatic transversality.
 \begin{theorem} \label{thm:localunstable}
Let $M$ be any Riemann surface. Let $\gamma \in \Omega
\Ham(M, \omega)$ be a closed smooth positive Morse index $k$,
$L ^{+} $ Ustilovsky geodesic.  
 If $ (f_\gamma, E ) $ is a local unstable manifold for $\gamma $ then
$$0 \neq [f _{\gamma}] \in  \pi _{k} ( \Omega \Ham(M, \omega),
\Omega   \Ham(M, \omega) _{E} ).$$
If the Morse index of $\gamma$ is 0 and $\gamma$ is quasi-integrable, then $\gamma$
is $L ^{+} $ length minimizing. This second part of the theorem holds
for a general symplectic manifold (assuming virtual techniques.)
\end{theorem}
The index 0 case of the above is just a very minor variation of a
foundational result in McDuff-Slimowitz
\cite{citeMcDuffSlimowitzHofer--ZehndercapacityandlengthminimizingHamiltonianpaths},
and must be well known to experts. A relatively elementary proof can be
given 
for example by modelling the proof of Theorem 1.9 in
 \cite{citeSavelyevMorsetheoryfortheHoferlengthfunctional} in index 0
 case. In what follows we outline the proof of the positive index
 case.
 \begin{proof} (Outline) 
The main difference here with the proof of  
\cite[Theorem
1.9]{citeSavelyevMorsetheoryfortheHoferlengthfunctional}, is that
we deal with closed loops, and that we can use in the context of
surfaces a stronger form of the ``automatic transversality''
\cite[Theorem
1.20]{citeSavelyevMorsetheoryfortheHoferlengthfunctional}. (It
becomes more automatic.)
The setup with closed loops here has already been done 
in the author's
\cite{citeSavelyevVirtualMorsetheoryon$Omega$Ham$(Momega)$.}, but without
the generality of Ustilovsky geodesics, and without the rather helpful ``index
theorem'' \ref{thm:indexUst}.
   
Given our local unstable manifold $(f _{\gamma}, E) $, we obtain a
Hamiltonian fibration $$M \hookrightarrow
P _{f _{\gamma} } \to B ^{k} \times S ^{2},  
$$
by doing the standard clutching construction $$M \times D ^{2} \sqcup
_{\gamma _{b} } M \times D ^{2},  $$  for a loop of diffeomorphisms of
$M$, $\gamma _{b} = f _{\gamma}(b)  $, as in the proof of Theorem
\ref{thm:contractibleloops},
parametrically as ${b}$   varies in $B ^{k} $.
For each $b$ we obtain a natural coupling form $ \widetilde{\Omega}_{b}
$ on $X _{b} \equiv P _{f _{\gamma} } |
_{\{b\} \times S ^{2} }$, as in the proof of Theorem
\ref{thm:contractibleloops}. From construction it is immediate that the family
$\{\widetilde{\Omega}_{b}\}$ extends to a closed form
$\widetilde{\Omega}  $
on $P _{f _{\gamma} } $ and that this is a coupling form. Let $\{J _{b} \}$
denote the family of almost complex structures induced by the family
$\{\widetilde{\Omega}_{b}\}$ as in the proof of Theorem
\ref{thm:contractibleloops}. The fixed point $x _{\max} $ of $\gamma$
gives a tautological flat and hence holomorphic section $\sigma _{\max} $ of $X _{0}  \to
S ^{2} $, let us call by $A$ its homology class in $P _{f _{\gamma} }
$.


Let $\overline{\mathcal{M}}(P _{f _{\gamma}}, A  ) $ denote the compactified
moduli space of pairs $(u,b)$, for  $u$ a  $J _{b} $-holomorphic
section of $X _{b}  $ in class $A$. Then by
the same argument that is used in
\cite[Proposition
1.19]{citeSavelyevMorsetheoryfortheHoferlengthfunctional},
we get that $\overline{\mathcal{M}}(P _{f _{\gamma}}, A  ) $ consists of a
single point $(\sigma _{\max},0 )$. Let us briefly give an idea for
this.
 For a given parameter $b$,  by a 
kind of classical energy positivity, a  $J _{b} $ holomorphic section
of $X _{b} $ in class $A$ will give a lower bound for the positive Hofer length of the loop
$f _{\gamma}(b) $ and this lower bound is exactly $L ^{+}(\gamma) $.
So all the  elements of our moduli space must appear in $X _{0}   $.
Then a neat but simple trick, also based on energy positivity and originally due
to Paul Seidel, allows one to conclude that $\sigma _{\max} $ is the
only possibility for such an element.

Let \begin{equation*} D  ^{rest}_{\sigma_{\max}}: \Omega ^{0} ( \sigma _{\max} ^{*}
      T ^{vert} X _{0})     \to \Omega ^{0,1} (\sigma
   _{\max} ^{*} T ^{vert} X _{0}  ), 
\end{equation*}
be the associated real linear Cauchy-Riemann operator. Then we claim
that this
operator has trivial kernel. This works as follows.  Twice of the vertical
Chern number of the normal bundle $\sigma _{\max} ^{*}
T ^{vert} X _{0}$ is less than $-2$, as by construction 
this is the winding (Maslov) number of the linearization of
$\gamma$ at $x _{\max} $, and by assumption that the Morse index
of $\gamma$ is positive and by the ``index'' theorem
\ref{thm:indexUst}, this winding number is less than $ -2$. Given this
by the (proof of) the first part of \cite[Theorem
1.20]{citeSavelyevMorsetheoryfortheHoferlengthfunctional}
 and by the observation
on the vanishing of the kernel of the CR operator above,
$\overline{\mathcal{M}}(P _{f _{\gamma}}, A  )$ can be regularized in such a way
that it still consists only of $\{\sigma _{\max} \}$. In particular
the associated parametric Gromov-Witten invariant is $\pm 1$, with the
sign depending on how one chooses to orient the local unstable manifold.
\begin{remark}
Parametric
here is just emphasizing that we are not counting holomorphic curves
in a symplectic manifold, but in a total space of a family of
symplectic manifolds, but it is still a count in the usual sense of
counting (with signs)
zeros of a section of an associated Banach bundle. We refer the reader
to \cite{citeSavelyevQuantumcharacteristicclassesandtheHofermetric.}
for further elaboration.
\end{remark}
The
theorem follows, as we cannot deform the cycle $f _{\gamma} $ to a
cycle $f' _{\gamma} $ lying
in any lower sublevel set as that would preclude any $A$ class
elements $(u,b)$
from existing in  $\overline{\mathcal{M}}(P _{f' _{\gamma}}, A  )$, since as we
indicated above such an element gives a lower bound of exactly $L ^{+}(\gamma) $
on $L ^{+}(f' _{\gamma} (b))$ which would be absurd.
\end{proof}
\begin{proof} [Proof of Theorem \ref{thm:dynamical}] 
   Let's first do
   the case of $\ham$. Suppose not, then we first observe that $\gamma$ must be contractible as the
   minimal  $L ^{+} $ length of a non-contractible geodesic is $\hbar/2$.
   This is well known but can itself be deduced from Theorem
   \ref{thm:localunstable}, as the Hamiltonian $S ^{1} $ action on $S
   ^{2} $ representing generator of $\pi_1 (\ham)$ satisfies the
   hypothesis, has positive Hofer length $\hbar/2$, and must be minimizing by the
   theorem.
   Next  suppose that $\gamma$ is a contractible quasi-integrable Ustilovsky
   geodesic and that the Morse index  of
   $\gamma$ is 0, in this case by Theorem \ref{thm:localunstable} 
   $\gamma$ is minimizing, which is absurd unless $\gamma$ is constant, so $\gamma$ has positive
   Morse index $k$.
   %
Then again by Theorem \ref{thm:localunstable}:
\begin{equation*} \pi _{k} (\Omega ^{\hbar/2 }
   \ham, \Omega \ham _{E}  ) \neq
   0, 
\end{equation*} for some $E < L ^{+}(\gamma) < \hbar/2$.
Then using Theorem
 \ref{thm:contractibleloops}, 
we get that the inclusion map of $\Omega \ham _{E}$ into
$\Omega ^{\hbar/2} \ham $ is contractible from which it follows that
\begin{equation*} \pi _{k} (\Omega ^{\hbar/2} 
   \ham) \neq
   0, 
\end{equation*} but this contradicts Theorem \ref{thm:contractibleloops}.

For the case of $M=\Sigma$, if the index of a quasi-integrable Ustilovsky
geodesic $\gamma$ is 0 then again by
Theorem \ref{thm:localunstable} $\gamma$ is minimizing which can only happen
if $\gamma$  is constant, as $\pi _{1} (\hams)=0 $. 

For a general $L ^{+} $ Ustilovsky geodesic
$\gamma$ with positive Morse index $k$ we get that
\begin{equation*} \pi _{k} (\Omega  
   \hams, \Omega   \hams _{E}  ) \neq
   0, 
\end{equation*} for some $E < L ^{+}(\gamma)$,
then using
Theorem \ref{thm:contractibleloops}, 
we get  that
\begin{equation*} \pi _{k} (\Omega  
   \hams) \neq
   0, 
\end{equation*} but this is impossible since 
$\hams$ is contractible.
 \end {proof}
\section {Acknowledgements}
The paper was completed at ICMAT Madrid,
while the author was a research fellow, with funding from Severo Ochoa grant and
Viktor Ginzburg laboratory. I am grateful to Alexander Shnirelman for
kindly listening to some ideas on the subject, and helpful comments on
the analogous situations in $L ^{2} $ geometry.
  To Dusa McDuff for interest, and much help in
 in clarifying the paper. Leonid Polterovich for reference on rough $CAT(0)$
 spaces, and some related post discussions. Viktor Ginzburg  for
 interest and 
 comments. As well as the anonymous referees for wonderful reading and
 catching some serious issues.
 \bibliographystyle{siam}  
 \bibliography{/root/texmf/bibtex/bib/link} 
\end{document}